\documentclass[11pt]{article}
\usepackage{amsthm,amsfonts,amssymb,amsmath}
\usepackage{stmaryrd}
\usepackage[pdftex,colorlinks,backref=page,citecolor=blue,bookmarks=false]{hyperref}
\usepackage{cleveref}
\usepackage{epsfig}
\usepackage{url}
\usepackage{xcolor}
\usepackage{multicol,graphicx}
\usepackage{fullpage}
\usepackage{thm-restate}
\usepackage{tocloft}
\usepackage{tikz, tikz-3dplot}
\usetikzlibrary{decorations.pathreplacing,calligraphy, patterns}
\usepackage[square,sort,comma,numbers]{natbib} 
\usepackage{setspace}
\usepackage[shortlabels]{enumitem}

\numberwithin{equation}{section}

\newtheorem{thm}[equation]{Theorem}

\newtheorem{lemma}[equation]{Lemma}
\newtheorem{prop}[equation]{Proposition}

\crefname{thm}{Theorem}{Theorems}
\crefname{cor}{Corollary}{Corollaries}
\crefname{lemma}{Lemma}{Lemmas}
\crefname{prop}{Proposition}{Propositions}
\crefname{claim}{Claim}{Claims}
\crefname{obs}{Observation}{Observations}
\crefname{fact}{Fact}{Facts}
\crefname{conj}{Conjecture}{Conjectures}
\crefname{ques}{Question}{Questions}
\crefname{prob}{Problem}{Problems}

\theoremstyle{definition}

\newtheorem*{ack}{Acknowledgements}

\theoremstyle{remark}

\newcommand{\mc}{\mathcal}

\newcommand{\beq}{\begin{equation*}}
\newcommand{\eeq}{\end{equation*}}

\title{Large Independent Sets in Flag Spheres}
\author{Varun Shah \footnotemark[1]
}

\date{}

\begin{document}

\maketitle
\begin{abstract}
	\setlength{\parskip}{\medskipamount}
    \setlength{\parindent}{0pt}
    \noindent

	For every $d \geq 4$, we construct a family of $(d-1)$-dimensional flag simplicial spheres $\mc K_n$ whose graphs contain independent sets of size asymptotically equal to the number of vertices. More precisely, we prove that for all sufficiently large $n$,
    \beq \alpha(G(\mc K_n)) \geq f_0(\mc K_n) - \frac{Cf_0(\mc K_n)}{\left(\log f_0(\mc K_n)\right)^{\lfloor d/2 \rfloor-1}}, \eeq
    where $C = C(d) > 0$. This disproves a recent conjecture of Chudnovsky and Nevo \cite{nevo}. 
\end{abstract}

\renewcommand{\thefootnote}{\fnsymbol{footnote}} 

\footnotetext[1]{Department of Mathematics, University of Washington, Seattle, WA 98195, USA. Email: 
  \textsf{\href{mailto:varunsh@uw.edu}{varunsh@uw.edu}.
           }}

\section{Introduction}
A simplicial complex $\mc K$ is \emph{flag} provided every minimal non-face of $\mc K$ has two vertices. The condition of flagness often imposes strong combinatorial restrictions on the structure of simplicial spheres. For instance, the boundary of the $d$-simplex has $\binom{d+1}{i+1}$ faces of dimension $i$, the smallest among all $(d-1)$-spheres. In contrast, Meshulam \cite{meshulam} proved that every flag simplicial $(d-1)$-sphere must satisfy $f_i \geq 2^{i+1}\binom{d}{i+1}$ for all $i$, with equality achieved by the boundary of the $d$-dimensional cross-polytope.

Another interesting comparison concerns the graphs of these spheres. A connected graph is said to be \emph{$k$-connected} if it has more than $k$ vertices and remains connected whenever fewer than $k$ vertices are removed. The \emph{connectivity} of a graph is the largest $k$ for which the graph is $k$-connected. A classical result of Balinski \cite{balinski}, later extended by Barnette \cite{barnette}, states that graphs of $(d-1)$-spheres are $d$-connected, while Athanasiadis \cite{athanasiadis} proved that the graph of every flag $(d-1)$-sphere has connectivity at least $2d-2$.

In this paper, we study the independence numbers of these graphs. The \emph{independence number} $\alpha(G)$ of a graph $G$ is the size of its largest independent set. Chudnovsky and Nevo \cite[Conjecture~4.1]{nevo} conjectured that if $\mc K$ is a flag $(d-1)$-sphere, then its graph $G(\mc K)$ satisfies 
\begin{equation} \label{eq: main.eq}
    \alpha(G(\mc K)) \leq \left\lfloor\frac{f_0(\mc K)-2(d-2)}{2}\right\rfloor.
\end{equation}
The conjecture is immediate when $d=2$ and was established by the authors of \cite{nevo} for $d=3$; however, their arguments rely essentially on planarity. They also note that equality in \eqref{eq: main.eq} is attained by the $(d-1)$-sphere obtained by suspending an $(n-2(d-2))$-gon $d-2$ times. 

If the spheres are not required to be flag, then the situation is seemingly different. The number of facets $m$ of a neighborly simplicial $(d-1)$-sphere with $n$ vertices (such as the boundary of the cyclic polytope $C(n, d)$) is on the order of $n^{\lfloor d/2\rfloor}$; when $d \geq 4$, $n/m$ approaches $0$ as $n \to \infty$. Performing stellar subdivisions at each facet produces a simplicial sphere with $n + m \sim m$ vertices, while the $m$ added vertices form an independent set. This shows that there exist simplicial spheres with independent sets of size asymptotically equal to the number of vertices. Since stellar subdivisions do not preserve flagness, however, this construction does not apply in the flag setting.

In this paper, we disprove the conjecture of Chudnovsky and Nevo for all $d \geq 4$ by establishing the following result.
\begin{thm} \label{thm: main.thm}
    There exists a constant $C = C(d)>0$ and a family of $(d-1)$-dimensional flag spheres $\mc K_n$ satisfying $f_0(\mc K_n) \to \infty$ as $n \to \infty$, and
    \beq \alpha(G(\mc K_n)) \geq f_0(\mc K_n) - \frac{Cf_0(\mc K_n)}{(\log f_0(\mc K_n))^{\lfloor d/2 \rfloor - 1}}, \eeq
    for all sufficiently large $n$. 

    When $d \geq 4$, taking $n \to \infty$ gives
    \beq \lim\limits_{n \to \infty} \frac{\alpha(G(\mc K_n))}{f_0(\mc K_n)} = 1. \eeq
\end{thm}
If true, \eqref{eq: main.eq} would imply that $\alpha(G(\mc K)) \leq \left\lfloor\frac{f_0(\mc K)}{2}\right\rfloor$ for all flag $(d-1)$-spheres $\mc K$, and consequently the above limit could not exceed $1/2$. So Theorem~\ref{thm: main.thm} provides a counterexample to \eqref{eq: main.eq} for every $d \geq 4$. 

Let us briefly describe our construction. Our starting point is the construction of (not necessarily flag) spheres with large independence number discussed above. We replace the simplicial spheres in that construction by cubical spheres with $n$ vertices and $m$ facets, where $n/m \to 0$ as $n \to \infty$. We then perform an operation resembling a stellar subdivision: each cubical facet is triangulated after adding a new vertex in its interior. This produces a flag sphere with $m+n \sim m$ vertices, while the $m$ new vertices form an independent set.

\textbf{Organization of the paper.} In Section~\ref{sec: bg}, we introduce necessary background on simplicial and cubical complexes, which are the main objects we use in this paper. In Section~\ref{sec: cubes}, we describe a triangulation of cubes, which is then used in Section~\ref{sec: main} to prove Theorem~\ref{thm: main.thm}.

\section{Background} \label{sec: bg}

\subsection{Simplicial Complexes} \label{sec: bgflag}

A \emph{simplicial complex} $\mc K$ on a vertex set $V = V(\mc K)$ is a collection of subsets of $V$ such that $\{v\} \in \mc K$ for all $v \in V$, and $F \subseteq G \in \mc K$ implies $F \in \mc K$. Elements of $\mc K$ are called \emph{faces}, and maximal faces (with respect to inclusion) are called \emph{facets}. The \emph{dimension} of a face $F$ is $|F|-1$, and the dimension of $\mc K$ is the largest dimension of any face of $\mc K$. If every facet of $\mc K$ has the same dimension, then $\mc K$ is said to be \emph{pure}. Moreover, $0$-dimensional faces are called \emph{vertices} and $1$-dimensional faces are called \emph{edges}. For all $k \geq 0$, we denote by $f_k(\mc K)$ the number of $k$-dimensional faces of $\mc K$. 

We often omit the usual set notation for faces when no ambiguity can arise. For example, if $\{x\}$ is a vertex of $\mc K$, we denote it by $x$, and if $\{x,y,z\}$ is a face of $\mc K$, we denote it by $xyz$. 

Let $k\geq 0$. The \emph{$k$-skeleton} of $\mc K$ is the simplicial complex consisting of all faces of $\mc K$ of dimension at most $k$. We refer to the $1$-skeleton of $\mc K$ as the \emph{graph} of $\mc K$ and denote it by $G(\mc K)$.

Every simplicial complex $\mc K$ has a \emph{geometric realization} $|\mc K|$, obtained as a union of simplices in Euclidean space whose intersections reflect the intersections among the faces of $\mc K$. This space is well-defined up to homeomorphism. Throughout the paper, we will therefore blur the distinction between a face of $\mc K$ and the corresponding simplex of $|\mc K|$, viewing $\mc K$ as a triangulation of $|\mc K|$.

A \emph{simplicial $(d-1)$-sphere} is a simplicial complex whose geometric realization is homeomorphic to the sphere $\mathbb S^{d-1}$. When no confusion is likely, we will simply call such a complex a \emph{$(d-1)$-sphere}.

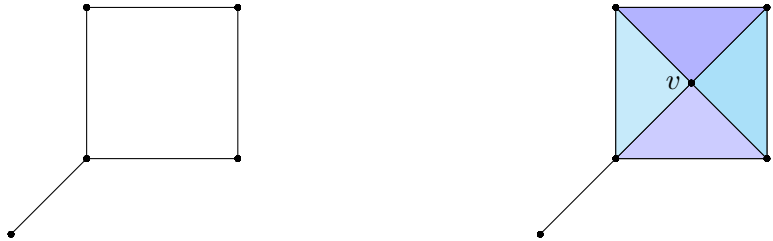
\begin{figure}[ht]
    \centering
    \begin{tikzpicture}[scale=2,line join=bevel]
    \begin{scope}[xshift=0cm]
    \coordinate (1) at (0,0);
    \coordinate (2) at (0,1);
    \coordinate (3) at (1,1);
    \coordinate (4) at (1,0);
    \coordinate (5) at (-0.5,-0.5);

    \draw (1) -- (2) -- (3) -- (4) -- cycle;
    \draw (1) -- (5);

    \draw[black] plot[mark=*, mark size=0.5] (1);
    \draw[black] plot[mark=*, mark size=0.5] (2);
    \draw[black] plot[mark=*, mark size=0.5] (3);
    \draw[black] plot[mark=*, mark size=0.5] (4);
    \draw[black] plot[mark=*, mark size=0.5] (5);
    \end{scope}
    
    \begin{scope}[xshift=3.5cm]
    \coordinate (1) at (0,0);
    \coordinate (2) at (1,0);
    \coordinate (3) at (1,1);
    \coordinate (4) at (0,1);
    \coordinate (5) at (-0.5,-0.5);
    \coordinate (v) at (0.5,0.5);

    \draw (1) -- (2) -- (3) -- (4) -- cycle;
    \draw (1) -- (5);
    \draw[fill=blue!20] (v) -- (1) -- (2) -- cycle;
    \draw[fill=cyan!30] (v) -- (2) -- (3) -- cycle;
    \draw[fill=blue!30] (v) -- (3) -- (4) -- cycle;
    \draw[fill=cyan!20] (v) -- (4) -- (1) -- cycle;

    \draw[black] plot[mark=*, mark size=0.5] (1);
    \draw[black] plot[mark=*, mark size=0.5] (2);
    \draw[black] plot[mark=*, mark size=0.5] (3);
    \draw[black] plot[mark=*, mark size=0.5] (4);
    \draw[black] plot[mark=*, mark size=0.5] (5);
    \draw[black] plot[mark=*, mark size=0.5] (v);
    \node[scale=1, left] at (v) {$v$};
    \end{scope}
    \end{tikzpicture}
    \caption{An example of coning: the complex on the right is the cone over the four vertices of the square.}
    \label{fig: cones}
\end{figure}

We say that a simplicial complex $\mc L$ is a \emph{subcomplex} of $\mc K$ if every face of $\mc L$ is a face of $\mc K$. If $U$ is a subset of the vertex set of $\mc K$, then the \emph{induced subcomplex} $\mc K[U]$ is the subcomplex of $\mc K$ with vertex set $U$ whose faces are all the faces of $\mc K$ contained in $U$. The \emph{cone of $\mc K$ over $U$} is the simplicial complex obtained by adjoining a new vertex $v$ (called the \emph{cone point}) to the vertex set of $\mc K$, and whose faces consist of the faces of $\mc K$ together with all faces of the form $F\cup v$, where $F$ is a face of $\mc K[U]$; see Figure~\ref{fig: cones}.

An important special case of coning is \emph{suspension}. The suspension of a simplicial complex $\mc K$ with vertex set $V$ is obtained by coning over $V$ twice. The suspension of a simplicial $(d-1)$-sphere is homeomorphic to $\mathbb S^d$ (see \cite[p.~8]{hatcher}).

\subsubsection{Flag Complexes} 

A simplicial complex is \emph{flag} if all of its minimal non-faces have two vertices; see Figure~\ref{fig: flag.comp} for an example. A set of pairwise adjacent vertices in a graph is a \emph{clique}. An equivalent characterization of the flag property is that every clique in $G(\mc K)$ is a face of $\mc K$.

\begin{figure}[ht]
    \centering
    \begin{tikzpicture}[scale=2,line join=bevel]
    \begin{scope}[xshift=0cm]
    \coordinate (A1) at (1,0,0);
    \coordinate (A2) at (-1,0,0);
    \coordinate (B1) at (0,1,0);
    \coordinate (B2) at (0,-1,0);
    \coordinate (C1) at (0,0,1);
    \coordinate (C2) at (0,0,-1);
    
    \draw (A1) -- (B1) -- (A2) -- (B2) -- cycle;
    \draw (C1) -- (A1);
    \draw (C1) -- (A2);
    \draw (C1) -- (B1);
    \draw (C1) -- (B2);
    \draw[dashed] (C2) -- (A1);
    \draw[dashed] (C2) -- (A2);
    \draw[dashed] (C2) -- (B1);
    \draw[dashed] (C2) -- (B2);
    \end{scope}

    \begin{scope}[xshift=3.5cm]
    \coordinate (B1) at (0,1,0);
    \coordinate (B2) at (0,-1,0);
    \coordinate (A1) at (0,0,0.5);
    \coordinate (A2) at ({-sqrt(3)/2},0,-0.5);
    \coordinate (A3) at ({sqrt(3)/2},0,-0.5);
    
    \draw[color=blue!30] (A3) -- (A1) -- (A2);
    \draw[dashed, color=blue!30] (A2) -- (A3);
    \draw (B1) -- (A1) -- (B2) -- (A2) -- cycle;
    \draw (B1) -- (A3) -- (B2) -- (A2) -- cycle; 
    \end{scope}
    \end{tikzpicture}
    \caption{The boundary of a cross-polytope is flag, whereas the boundary of a triangular bipyramid is not: the blue triangle is a clique but not a face.}
    \label{fig: flag.comp}
\end{figure}
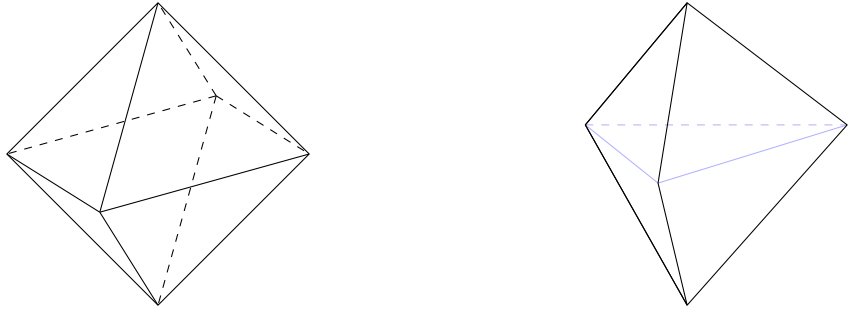

An important feature of flagness is that it is preserved under coning.

\begin{lemma} \label{lem: imp.lem}
Let $\mc K$ be a flag complex with vertex set $V$. The cone of $\mc K$ over any subset $U\subseteq V$ is also flag.
\end{lemma}

\begin{proof}
Let $v$ be the cone point, and let $A\subseteq V\cup v$ be a clique in the graph of the cone. If $v\notin A$, then $A$ is a clique in $G(\mc K)$, and since $\mc K$ is flag, $A$ is a face of $\mc K$. Hence $A$ is also a face of the cone.

Now suppose that $v\in A$. Then $A\setminus v$ is a clique in $G(\mc K)$, and therefore a face of $\mc K$. Moreover, $A\setminus v\subseteq U$, since $v$ is adjacent only to vertices of $U$. Thus $A\setminus v$ is a face of $\mc K[U]$, so $A$ is a face of the cone by definition.
\end{proof}

\subsection{Cubes}

The \emph{$n$-cube} is the polytope $[0,1]^n$. We refer the reader to \cite{ziegler} for background on polytopes and the combinatorics of cubes. We will only need the following basic fact: every $k$-dimensional face of the $n$-cube is obtained by fixing $n-k$ coordinates to be either $0$ or $1$ and allowing the remaining $k$ coordinates to vary freely. Consequently, every $k$-face is itself a $k$-cube and has $2^k$ vertices.

The graph $Q_n$ of $[0,1]^n$ has vertex set $\{0,1\}^n$, and two vertices $u$ and $v$ are adjacent if they differ in exactly one coordinate. In general, the \emph{distance} $d(u, v)$ between vertices $u, v$ is the number of coordinates in which they differ. Since every $k$-dimensional face of $[0,1]^n$ is itself a $k$-cube, its vertex set induces a subgraph of $Q_n$ isomorphic to $Q_k$. The converse, which appears for example in \cite[Lemma 12]{ehrenborg}, will be useful in Section~\ref{sec: main}. We include a proof for completeness.

\begin{lemma} \label{lem: subcubes}
    If $H$ is a subgraph of $Q_n$ isomorphic to $Q_k$, then $V(H)$ is the set of vertices of some $k$-dimensional face of $[0,1]^n$.
\end{lemma}

\begin{proof}
    We argue by induction on $k$. The cases $k=0$ and $k=1$ are immediate, so assume $k\geq 2$. The vertex set of $H$ can be partitioned into two sets $A$ and $B$, each of size $2^{k-1}$, such that the induced subgraphs on $A$ and $B$ are both isomorphic to $Q_{k-1}$. By the induction hypothesis, $A$ and $B$ are the vertex sets of $(k-1)$-dimensional faces of $[0,1]^n$. 

    Let $I$ be the set of coordinates that are fixed on $A$. We claim that among the coordinates in $I$, at most one is not fixed in $B$ to the same value as on $A$. Suppose otherwise. Then there exist two coordinates $i,j \in I$ such that $B$ is either free in those coordinates or fixed to the opposite value. Hence we may choose $b \in B$ that differs from the fixed values on $A$ in both coordinates $i$ and $j$.

    Since $A$ is the vertex set of a $(k-1)$-dimensional face, the $k-1$ coordinates not in $I$ vary freely in $A$. Choosing $a \in A$ so that it disagrees with $b$ in each of these varying coordinates, we obtain
    \beq d(a,b) \geq (k-1)+2 = k+1,\eeq
    contradicting the fact that every two vertices of $H$ are at distance at most $k$ in $Q_n$. Indeed, since $H$ is isomorphic to $Q_k$, any pair of vertices in $H$ is connected by a path of length at most $k$, and each edge of $Q_n$ changes exactly one coordinate.

    It follows that the vertices of $H$ vary freely in at most $k$ coordinates, and hence $V(H)$ is contained in the vertex set of a $k$-dimensional face $G$ of $[0,1]^n$. Since both $V(H)$ and $V(G)$ have cardinality $2^k$, $V(H)$ must equal $V(G)$.
\end{proof}

\subsection{Cubical Complexes}

A \emph{cubical complex} $\mc B$ on a vertex set $V$ is a collection of subsets of $V$, partially ordered by inclusion, satisfying the following properties:
\begin{enumerate}
\item $\varnothing\in\mc B$ and ${v}\in\mc B$ for every $v\in V$.
\item For every nonempty face $F\in\mc B$, the interval
\beq[\varnothing,F]=\{G\in\mc B:\varnothing\subseteq G\subseteq F\}\eeq
is isomorphic to the face poset of a cube of some dimension.
\item If $F,G\in\mc B$, then $F\cap G\in\mc B$.
\end{enumerate}

In particular, the collection of vertex sets of faces of an $n$-cube forms a cubical complex in this sense.

The elements of $\mc B$ are called \emph{faces}. If $F\in\mc B$ is nonempty, the \emph{dimension} of $F$ is the dimension of the cube whose face poset is isomorphic to $[\varnothing,F]$. The dimension of $\mc B$ is the largest dimension of any face of $\mc B$. As in the simplicial setting, maximal faces are called \emph{facets}, $0$-dimensional faces are called \emph{vertices}, and $1$-dimensional faces are called \emph{edges}. We denote by $f_k(\mc B)$ the number of $k$-dimensional faces of $\mc B$.

Many constructions for simplicial complexes have cubical analogues. In particular, the \emph{graph} of $\mc B$, denoted $G(\mc B)$, is the graph whose vertices are the vertices of $\mc B$ and whose edges are the $1$-dimensional faces of $\mc B$. 

Every cubical complex admits a geometric realization as a topological space, unique up to homeomorphism. A \emph{cubical $(d-1)$-sphere} is a cubical complex whose geometric realization is homeomorphic to $\mathbb S^{d-1}$.

\subsubsection{Neighborly Cubical Spheres}

A cubical $(d-1)$-sphere with $2^n$ vertices is said to be \emph{neighborly} if its $(\lfloor d/2\rfloor-1)$-skeleton is that of the $n$-cube. For every $d\geq 1$ and $n\geq d$, Babson, Billera, and Chan \cite{babson} constructed neighborly cubical $(d-1)$-spheres with $2^n$ vertices in 1997. Subsequently, Joswig and Ziegler \cite{joswig} gave constructions of such spheres that are realizable as boundaries of convex polytopes.

Let $d \geq 4$, and let $\mc P$ be a neighborly cubical $(d-1)$-sphere having $2^n$ vertices. Label the vertices of $\mc P$ by elements of $\{0,1\}^n$ so that the graph of $\mc P$ coincides with $Q_n$. Any $k$-dimensional face $F$ of $\mc P$ is combinatorially isomorphic to $[0,1]^k$, and hence the subgraph of $Q_n$ induced on $F$ is isomorphic to $Q_k$. By Lemma~\ref{lem: subcubes}, it follows that $F$ is the vertex set of a $k$-face of $[0,1]^n$. Thus $\mc P$ is a subcomplex of $[0,1]^n$, proving the following proposition. 

\begin{prop} \label{prop: cube.subcomp}
    A neighborly cubical $(d-1)$-sphere with $2^n$ vertices is a subcomplex of $[0,1]^n$ whenever $d\geq 4$.
\end{prop}

The face numbers of cubical spheres satisfy a version of the Dehn--Sommerville relations \cite{adin}. Combined with the structure of neighborly spheres, these imply that if $\mc P$ is a neighborly cubical $(d-1)$-sphere, then
\begin{equation}\label{eq: facet.bound}
M_1 f_0(\mc P)(\log f_0 (\mc P))^{\lfloor d/2\rfloor-1} \leq f_{d-1}(\mc P) \leq M_2 f_0(\mc P)(\log f_0(\mc P))^{\lfloor d/2\rfloor-1}
\end{equation}
for some constants $M_1, M_2 > 0$ depending only on $d$.

\section{Flag Triangulation of the Cube} \label{sec: cubes}

In this section, we describe a triangulation of the $(d-1)$-cube $[0,1]^{d-1}$ that will play a key role in the construction of the flag spheres $\mc K_n$. To do so, we first introduce some terminology concerning partially ordered sets; see \cite[Chapter~3]{stanley} for a detailed exposition.

Define a partial order on $\{0,1\}^{d-1}$ by declaring $x\leq y$ if $x_i\leq y_i$ for every $i$. The resulting poset is called the \emph{Boolean lattice}.

The faces of $[0,1]^{d-1}$ admit a natural description in terms of this partial order. In the previous section, we observed that the vertex set of any face $G$ is obtained by fixing a subset of coordinates to be $0$ or $1$, while letting the remaining coordinates vary freely. It follows that every face is an interval in this poset; let us briefly explain why. Indeed, its minimal element is obtained by setting all free coordinates to $0$, and its maximal element by setting them to $1$. In particular, $\{0,1\}^{d-1}$ is the interval $[\mathbf 0, \mathbf 1]$, where $\mathbf 0$ and $\mathbf 1$ denote the all-zero and all-one vertices, respectively. Henceforth, we will freely identify faces of the cube with such intervals, without further comment.

A set of elements of a poset is called a \emph{chain} if any pair of elements is comparable, or equivalently, if we can label its elements $x^{(1)}, \ldots, x^{(r)}$ so that 
\beq x^{(1)} \leq \cdots \leq x^{(r)}. \eeq
In this case we say that the chain has \emph{length} $r-1$. A \emph{maximal chain} is a chain that is not contained in any longer chain.

\begin{figure}[ht]
    \centering
    \begin{tikzpicture}[scale=2,line join=bevel]
    \begin{scope}[xshift=0cm, yshift=-0.3cm]
    \coordinate (A1) at (0,0);
    \coordinate (A2) at (0,1);
    \coordinate (A3) at (1,1);
    \coordinate (A4) at (1,0);

    \draw[fill=blue!30] (A1) -- (A4) -- (A3) -- cycle;
    \draw[fill=blue!20] (A1) --(A2) -- (A3) -- cycle;
    \draw[black] plot[mark=*, mark size=0.5] (A1);
    \draw[black] plot[mark=*, mark size=0.5] (A2);
    \draw[black] plot[mark=*, mark size=0.5] (A3);
    \draw[black] plot[mark=*, mark size=0.5] (A4);
    \node[scale=0.8, below left] at (A1) {$00$};
    \node[scale=0.8, above left] at (A2) {$01$};
    \node[scale=0.8, above right] at (A3) {$11$};
    \node[scale=0.8, below right] at (A4) {$10$};
    \end{scope}
    
    \begin{scope}[xshift=3.5cm]
    \coordinate (A1) at (0,0,0);
    \coordinate (A2) at (1.2,0,0);
    \coordinate (A3) at (1.2,0,1);
    \coordinate (A4) at (0,0,1);
    \coordinate (B1) at (0,1,0);
    \coordinate (B2) at (1.2,1,0);
    \coordinate (B3) at (1.2,1,1);
    \coordinate (B4) at (0,1,1);

    \node[scale=0.5, below] at (A1) {$001$};
    \draw (B1) -- (B2) -- (A2) -- (A3) -- (A4) -- (B4) -- cycle;
    \draw (B2) -- (B3) -- (B4);
    \draw (A3) -- (B3);
    \draw[dashed] (B1) -- (A1) -- (A2);
    \draw[dashed] (A1) -- (A4);
    \draw[dashed, fill=purple!30, fill opacity=0.5] (A4) -- (B3) -- (B2) -- cycle;
    \draw[dashed, fill=violet!30, fill opacity=0.5] (A3) -- (A4) -- (B2) -- cycle;
    \draw[fill=cyan!30, fill opacity=0.5] (A3) -- (A4) -- (B3) -- cycle;
    \draw[fill=blue!30, fill opacity=0.5] (A3) -- (B3) -- (B2) -- cycle;
    \draw[black] plot[mark=*, mark size=0.5] (A1);
    \draw[black] plot[mark=*, mark size=0.5] (A2);
    \draw[black] plot[mark=*, mark size=0.5] (A3);
    \draw[black] plot[mark=*, mark size=0.5] (A4);
    \draw[black] plot[mark=*, mark size=0.5] (B1);
    \draw[black] plot[mark=*, mark size=0.5] (B2);
    \draw[black] plot[mark=*, mark size=0.5] (B3);
    \draw[black] plot[mark=*, mark size=0.5] (B4);
    \node[scale=0.5, below right] at (A2) {$101$};
    \node[scale=0.5, below] at (A3) {$100$};
    \node[scale=0.5, below] at (A4) {$000$};
    \node[scale=0.5, above] at (B1) {$011$};
    \node[scale=0.5, above] at (B2) {$111$};
    \node[scale=0.5, above] at (B3) {$110$};
    \node[scale=0.5, above left] at (B4) {$010$};
    \end{scope}
    \end{tikzpicture}
    \caption{The staircase triangulation of the square and the simplex corresponding to the chain $000 \leq 100 \leq 110 \leq 111$ in the staircase triangulation of the $3$-cube.}
    \label{fig: staircase}
\end{figure}
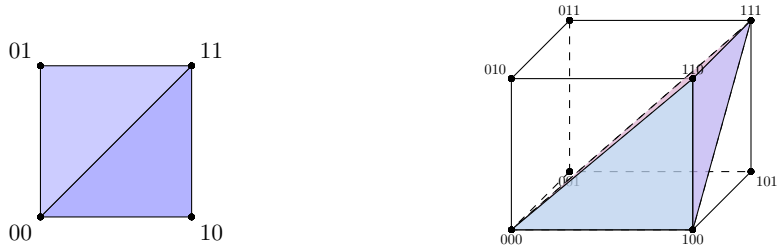

Let $\mc S$ be the simplicial complex with vertex set $\{0,1\}^{d-1}$ whose faces are the chains in the Boolean lattice (so that maximal chains are the facets). This determines a triangulation of the cube (see \cite[Section~6.3.2]{deloera}) known as the \emph{staircase triangulation} (see Figure~\ref{fig: staircase}). Since two vertices of $\mc S$ are adjacent precisely when they are comparable, every clique consists of pairwise comparable elements. Such a set is a chain, and hence a face of $\mc S$. Thus $\mc S$ is flag\footnote{In general, the simplicial complex of all chains of any poset $P$ is a flag complex called the \emph{order complex} of $P$.}.

Denote by $\partial \mc S$ the subcomplex of $\mc S$ consisting of simplices contained in the boundary of the cube. Clearly, $\partial \mc S$ triangulates the boundary of the cube.

\begin{prop} \label{prop: partial.flag}
    The complex $\partial \mc S$ is flag.
\end{prop}

\begin{proof}
    Let $A$ be a clique in the graph of $\partial \mc S$. Since $\mc S$ is flag, $A$ is a face of $\mc S$. All faces of $\mc S$ are chains in $\{0, 1\}^{d-1}$, so we may enumerate the elements of $A$ as $x^{(1)} \leq \cdots \leq x^{(r)}$.
    
    Moreover, $A$ cannot contain both $\mathbf 0$ and $\mathbf 1$, since these vertices are not adjacent in $\partial \mc S$. If $\mathbf 0 \notin A$, then $x^{(1)} \neq \mathbf 0$, and every element of $A$ lies in the interval $\left[x^{(1)}, \mathbf 1\right]$, which is a facet of $[0,1]^{d-1}$. Similarly, if $\mathbf 1 \notin A$, then $A \subseteq \left[\mathbf 0, x^{(r)}\right]$. Hence $A$ is contained in a facet of the cube and therefore is a face of $\partial \mc S$. Thus $\partial \mc S$ is flag.
\end{proof}

Let $v$ be an interior point of $[0,1]^{d-1}$. Coning over $\partial \mc S$ from $v$, we obtain a triangulation $\mc T$ of $[0,1]^{d-1}$, as illustrated in Figure~\ref{fig: tn}. By Lemma~\ref{lem: imp.lem} and Proposition~\ref{prop: partial.flag}, $\mc T$ is flag. The following lemma is immediate from the definitions.

\begin{lemma} \label{lem: faces.s}
    Let $F$ be a proper face of $[0,1]^{d-1}$. Then the faces of the subcomplex of $\mc T$ induced on $F$ are precisely the chains in the interval of the Boolean lattice corresponding to $F$.
\end{lemma}

\section{Proof of Theorem~\ref{thm: main.thm}} \label{sec: main}

In this section, we prove Theorem~\ref{thm: main.thm}. Let $\mc T$ be the flag triangulation of $[0,1]^{d-1}$ constructed in the previous section. 

Let $\mc P_n$ be a neighborly cubical $(d-1)$-sphere with $2^n$ vertices. When $d \geq 4$, the graph of $\mc P_n$ is the graph $Q_n$ of $[0,1]^n$. By Proposition~\ref{prop: cube.subcomp}, we may realize $\mc P_n$ as a subcomplex of $[0,1]^n$. In particular, any facet of $\mc P_n$ is a $(d-1)$-dimensional face of the cube, and so its vertices are obtained by fixing $n-(d-1)$ coordinates and allowing the remaining $d-1$ to vary freely. Forgetting the fixed coordinates identifies the vertices of the facet with $\{0,1\}^{d-1}$, thereby allowing us to transfer the triangulation $\mc T$ to each facet. 

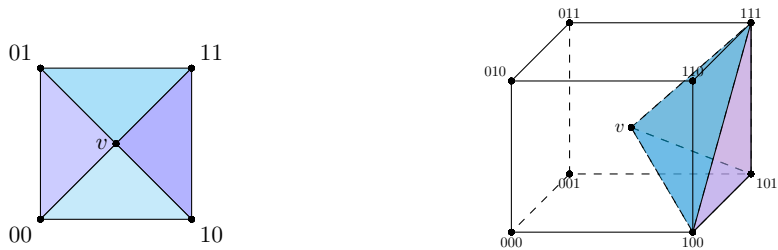
\begin{figure}[ht]
    \centering
    \begin{tikzpicture}[scale=2,line join=bevel]
    \begin{scope}[xshift=0cm, yshift=-0.3cm]
    \coordinate (A1) at (0,0);
    \coordinate (A2) at (0,1);
    \coordinate (A3) at (1,1);
    \coordinate (A4) at (1,0);
    \coordinate (B) at (0.5,0.5);

    \draw[fill=blue!20] (A1) -- (A2) -- (B) -- cycle;
    \draw[fill=cyan!30] (A2) -- (A3) -- (B) -- cycle;
    \draw[fill=blue!30] (A3) -- (A4) -- (B) -- cycle;
    \draw[fill=cyan!20] (A4) -- (A1) -- (B) -- cycle;
    \draw[black] plot[mark=*, mark size=0.5] (A1);
    \draw[black] plot[mark=*, mark size=0.5] (A2);
    \draw[black] plot[mark=*, mark size=0.5] (A3);
    \draw[black] plot[mark=*, mark size=0.5] (A4);
    \draw[black] plot[mark=*, mark size=0.5] (B);
    \node[scale=0.8, below left] at (A1) {$00$};
    \node[scale=0.8, above left] at (A2) {$01$};
    \node[scale=0.8, above right] at (A3) {$11$};
    \node[scale=0.8, below right] at (A4) {$10$};
    \node[scale=0.8, left] at (B) {$v$};
    \end{scope}
    
    \begin{scope}[xshift=3.5cm]
    \coordinate (A1) at (0,0,0);
    \coordinate (A2) at (1.2,0,0);
    \coordinate (A3) at (1.2,0,1);
    \coordinate (A4) at (0,0,1);
    \coordinate (B1) at (0,1,0);
    \coordinate (B2) at (1.2,1,0);
    \coordinate (B3) at (1.2,1,1);
    \coordinate (B4) at (0,1,1);
    \coordinate (C) at (0.6,0.5,0.5);

    \draw (B1) -- (B2) -- (A2) -- (A3) -- (A4) -- (B4) -- cycle;
    \draw[dashed] (B1) -- (A1) -- (A2);
    \draw[dashed] (A1) -- (A4);
    \draw[dashed, fill=purple!50, fill opacity=0.5] (C) -- (B2) -- (A2) -- cycle;
    \draw[dashed, fill=violet!50, fill opacity=0.5] (C) -- (A2) -- (A3) -- cycle;
    \draw[dashed, fill=cyan!80, fill opacity=0.5] (C) -- (A3) -- (B2) -- cycle;
    \draw[fill=blue!30, fill opacity=0.5] (B2) -- (A2) -- (A3) -- cycle;
    \draw (B2) -- (B3) -- (B4);
    \draw (A3) -- (B3);
    \draw[black] plot[mark=*, mark size=0.5] (A1);
    \draw[black] plot[mark=*, mark size=0.5] (A2);
    \draw[black] plot[mark=*, mark size=0.5] (A3);
    \draw[black] plot[mark=*, mark size=0.5] (A4);
    \draw[black] plot[mark=*, mark size=0.5] (B1);
    \draw[black] plot[mark=*, mark size=0.5] (B2);
    \draw[black] plot[mark=*, mark size=0.5] (B3);
    \draw[black] plot[mark=*, mark size=0.5] (B4);
    \draw[black] plot[mark=*, mark size=0.5] (C);
    \node[scale=0.5, below] at (A1) {$001$};
    \node[scale=0.5, below right] at (A2) {$101$};
    \node[scale=0.5, below] at (A3) {$100$};
    \node[scale=0.5, below] at (A4) {$000$};
    \node[scale=0.5, above] at (B1) {$011$};
    \node[scale=0.5, above] at (B2) {$111$};
    \node[scale=0.5, above] at (B3) {$110$};
    \node[scale=0.5, above left] at (B4) {$010$};
    \node[scale=0.65, left] at (C) {$v$};
    \end{scope}
    \end{tikzpicture}
    \caption{The triangulation $\mc T$ of the square and a simplex of the triangulation $\mc T$ of the cube.}
    \label{fig: tn}
\end{figure}

This produces a well-defined triangulation $\mc K_n$ of $\mc P_n$. Indeed, it suffices to check compatibility on intersections of facets. For a facet $F$ of $\mc P_n$, let $\mc T(F)$ denote the copy of $\mc T$ used to replace $F$. 

Suppose $F$ and $F'$ are facets of $\mc P_n$ that intersect in a face $G$. By Proposition~\ref{prop: cube.subcomp}, $F$, $F'$ and $G$ are faces of $[0,1]^n$, and hence may be regarded intervals in the Boolean lattice $\{0,1\}^n$. By Lemma~\ref{lem: faces.s}, both $\mc T(F)$ and $\mc T(F')$ induce on $G$ the triangulation whose faces are given by chains in $\{0,1\}^n$ contained in $V(G)$. Thus the induced triangulations of $G$ coincide, and the triangulations on the facets are compatible. 

\begin{lemma} \label{lem: subparts}
    If $x, y \in \{0,1\}^n$ such that $xy$ is an edge of $\mc K_n$, then
    \begin{enumerate}[label={(\roman*)},itemindent=1em]
        \item $x, y \in F$ for some facet $F$ of $\mc P_n$,  \label{lem: infacet}
        \item $x$ and $y$ are comparable, i.e. $x \leq y$ or $y \leq x$. \label{lem: compare}
    \end{enumerate}
\end{lemma}

\begin{proof}
    If $xy$ is an edge of $\mc K_n$, then $x$ and $y$ lie in some facet of $\mc K_n$. By construction, every facet of $\mc K_n$ is contained in some $\mc T(F)$, where $F$ is a facet of $\mc P_n$. Since $V(\mc T(F))$ consists of the vertices of $F$ together with a single additional coning vertex, and since $x,y \in {0,1}^n$ are not coning vertices, it follows that $x, y \in V(F)$. This proves \ref{lem: infacet}. 

    For \ref{lem: compare}, observe that every face of $\mc T(F)$ whose vertices are all contained in $\{0,1\}^n$ is a chain in $\{0,1\}^n$. Since $xy$ is such a face, $\{x, y\}$ is a chain, and hence either $x \leq y$ or $y \leq x$. 
\end{proof}

\begin{prop} \label{prop: kn.flag}
    The simplicial complex $\mc K_n$ is a flag $(d-1)$-sphere.
\end{prop}

\begin{proof}
Let $A$ be a clique in the graph of $\mc K_n$. We show that $A$ is a face of $\mc K_n$.

We first prove that $A \subseteq V(\mc T(F))$ for some facet $F$ of $\mc P_n$. The vertex set of $\mc K_n$ consists of the vertices of $\mc P_n$ together with one cone vertex for each facet of $\mc P_n$. If $A$ contains a cone vertex $v_F$ of some $\mc T(F)$, then by construction $v_F$ is adjacent only to vertices of $F$. Hence $A \subseteq V(\mc T(F))$.

Now suppose that $A$ consists entirely of original vertices of $\mc P_n$. If $x,y$ are adjacent in $G(\mc K_n)$, then by Lemma~\ref{lem: subparts}, item~\ref{lem: compare}, they are comparable. Hence the elements of $A$ are pairwise comparable, so $A$ is a chain. Write
\beq A = \{x^{(1)} \leq \cdots \leq x^{(r)}\}.\eeq
Since $A$ is a clique, $x^{(1)}$ and $x^{(r)}$ are adjacent in $G(\mc K_n)$. By Lemma~\ref{lem: subparts}, item~\ref{lem: infacet}, there exists a facet $F$ of $\mc P_n$ containing both $x^{(1)}$ and $x^{(r)}$. Since $F$ is an interval in the Boolean lattice, it contains the entire interval $\left[x^{(1)}, x^{(r)}\right]$, and hence $A \subseteq F$.

In either case, we obtain $A \subseteq V(\mc T(F))$ for some facet $F$ of $\mc P_n$. Since each $\mc T(F)$ is flag, it follows that $A$ is a face of $\mc T(F)$, and hence also of $\mc K_n$. Therefore $\mc K_n$ is flag.

Moreover, $\mc K_n$ triangulates the $(d-1)$-sphere $\mc P_n$, which implies that $\mc K_n$ is a flag $(d-1)$-sphere.
\end{proof}

Now we can prove our main theorem.

\begin{proof}[Proof of Theorem~\ref{thm: main.thm}]
    By Proposition~\ref{prop: kn.flag}, the complex $\mc K_n$ is a flag $(d-1)$-sphere for every $n \geq 1$. Its vertex set consists of $\mc P_n$ together with a cone vertex for each facet of $\mc P_n$, and hence
    \beq f_0(\mc K_n)=f_0(\mc P_n)+f_{d-1}(\mc P_n). \eeq
    Each cone vertex is adjacent only to original vertices. Hence the set of cone vertices is independent, and $\alpha(G(\mc K_n))\geq f_{d-1}(\mc P_n)$. Therefore
    \begin{equation} \label{eq: ugly}
    \frac{\alpha(G(\mc K_n))}{f_0(\mc K_n)} \geq \frac{f_{d-1}(\mc P_n)}{f_0(\mc P_n)+f_{d-1}(\mc P_n)} = 1-\frac{f_0(\mc P_n)}{f_0(\mc P_n)+f_{d-1}(\mc P_n)} \geq 1-\frac{f_0(\mc P_n)}{f_{d-1}(\mc P_n)}. 
    \end{equation}
    
    Recall that $f_0(\mc P_n)=2^n$, and so, by \eqref{eq: facet.bound}, there exist constants $M_1,M_2>0$ such that 
    \beq M_1n^{\lfloor d/2 \rfloor-1}2^n\leq f_{d-1}(\mc P_n)\leq M_2n^{\lfloor d/2 \rfloor-1}2^n. \eeq 
    The upper bound implies $f_0(\mc K_n)\leq 2^n(1+M_2n^{\lfloor d/2 \rfloor-1})\leq e^n$ for sufficiently large $n$, and hence $\log f_0(\mc K_n) \leq n$. Combining this with the lower bound on $f_{d-1}(\mc P_n)$ gives
    \beq \frac{f_0(\mc P_n)}{f_{d-1}(\mc P_n)} \leq \frac{1}{M_1n^{\lfloor d/2 \rfloor-1}} \leq \frac{1}{M_1(\log f_0(\mc K_n))^{\lfloor d/2 \rfloor-1}}. \eeq
    Substituting into \eqref{eq: ugly} and multiplying through by $f_0(\mc K_n)$ gives
    \beq \alpha(G(\mc K_n)) \geq f_0(\mc K_n)-\frac{Cf_0(\mc K_n)}{(\log f_0(\mc K_n))^{\lfloor d/2 \rfloor-1}},\eeq
    where $C=1/M_1$.
\end{proof}

\begin{ack}
The author thanks Isabella Novik for her invaluable guidance throughout this project and for her many helpful comments and suggestions on the manuscript. The author also thanks Eran Nevo for suggesting the use of staircase triangulations to strengthen the previous version of Theorem~\ref{thm: main.thm} in higher dimensions. The author was partially supported by a graduate fellowship from NSF grant DMS-2246399.
\end{ack}

\bibliographystyle{alpha}
\bibliography{refs.bib}
\end{document}